\documentclass{amsart}
\usepackage[applemac]{inputenc}\usepackage[french, english]{babel}\usepackage{amssymb,amsmath,mathrsfs}\usepackage[all]{xy} \usepackage[T1]{fontenc}
\usepackage{enumitem}

\DeclareMathOperator{\codim}{codim}
\DeclareMathOperator{\Gal}{Gal}

\DeclareMathOperator{\lcm}{lcm}

\DeclareMathOperator{\Stab}{Stab}

\newcommand{\Ators}{A_{\mathrm{tors}}}

\newcommand{\PP}{{\mathbb{P}}}

\newtheorem{theorem}{Theorem}[section]

\newtheorem{lemma}[theorem]{Lemma}
\newtheorem{conjecture}[theorem]{Conjecture}
\newtheorem{proposition}[theorem]{Proposition}
\newtheorem{definition}[theorem]{Definition}
\newtheorem{corollary}[theorem]{Corollary}

\begin{document}

\title{A bound for the torsion on subvarieties of abelian varieties}

\author{Aur\'elien Galateau}
\address{Universit\'e de Bourgogne Franche-Comt\'e}
\email{aurelien.galateau@univ-fcomte.fr}
\author{C\'esar Mart\'inez}
\address{Universit\'e de Caen}
\email{cesar.martinez@unicaen.fr}

\begin{abstract}
We give a uniform bound on the degree of the maximal torsion cosets for subvarieties of an abelian variety. The proof combines algebraic interpolation and a theorem of Serre on homotheties in the Galois representation associated to the torsion subgroup of an abelian variety. 
\end{abstract}

\maketitle

\section{Introduction}
The problem of understanding the distribution of torsion points in subvarieties of abelian varieties was independently raised by Manin and Mumford, who stated the following conjecture.
\begin{conjecture}[Manin-Mumford]
Let $C$ be an algebraic curve of genus $g \geq 2$, defined over a number field and embedded in its Jacobian~$J$. The set of torsion points of $J$ which lie in $C$ is finite. 
\end{conjecture}

\noindent
It was proved in 1983 by Raynaud~\cite{Raynaud83}, who soon generalized his theorem to arbitrary subvarieties of an abelian variety.
\begin{theorem}[Raynaud, \cite{Raynaud83a}]
\label{raynaud}
Let $V$ be a subvariety of an abelian variety $A$ defined over a number field. Then the Zariski closure of the set of torsion points in $V$ is a finite union of translates of abelian subvarieties of $A$ by torsion points. 
\end{theorem}

\noindent
This statement is still true if $A$ is replaced by a semi-abelian variety defined over a number field (see~\cite{Hindry88}) or even any field of caracteristic~$0$ (see~\cite{McQuillan95}). The aim of this paper is to give a precise bound for the degree of the maximal {\it torsion cosets} (translates of abelian subvarieties by torsion points) that appear in Raynaud's theorem.

\subsection{The number of torsion points on curves}

In the case of curves, the most spectacular quantitative versions of Raynaud's theorem are related to a famous conjecture of Coleman.
Let $C$ be a curve of genus $g \geq 2$ embedded in its Jacobian $J$, all defined over a number field $K$. We denote by $J_{\mathrm{tors}}$ the torsion subgroup of $J$ and by $C_{\mathrm{tors}}:= C(\bar{K}) \cap J_{\mathrm{tors}}$ the finite torsion subset of~$C$.

\begin{conjecture}[Coleman, \cite{Coleman87}]
\label{coleman}
Let $\mathfrak{p}$ be a prime ideal of $\mathcal{O}_K$ above a rational prime $p$ such that the following conditions are satisfied:

- $p \geq 5$,

- $K/\mathbb{Q}$ is unramified at $\mathfrak{p}$,

- $C$ has good reduction at $\mathfrak{p}$.

\noindent
Then the extension $K(C_{\mathrm{tors}})/K$ is unramified above $\mathfrak{p}$.
\end{conjecture}

\noindent
Using $p$-adic integration theory, Coleman managed to prove his conjecture in several significant cases, for instance if $p \geq 2g+1$ or if $C$ has ordinary reduction at~$\mathfrak{p}$. The ramification properties of the field generated by $C_{\mathrm{tors}}$ over $K$ provide valuable information in view of a quantitative version of the Manin-Mumford conjecture.  

\begin{theorem}[Coleman, \cite{Coleman85}]
Assume that $\mathfrak{p}$ and $p$ satisfy the hypotheses of Conjecture \ref{coleman}. If $C$ has ordinary reduction at $\mathfrak{p}$ and $J$ has (potential) complex multiplication, then $$|C_{\mathrm{tors}}| \leq pg.$$
\end{theorem}

\noindent
The assumptions of the theorem are needed to get a bound of this strength, which is sharp (\cite{Boxall-Grant00,Coleman85}). Using $p$-jets and Coleman's work, Buium gave an almost unconditional estimate, which can be expressed in the following (slightly weaker) form.
\begin{theorem}[Buium, \cite{Buium96}]
If $\mathfrak{p}$ and $p \geq 2g+1$ satisfy the hypotheses of Conjecture~\ref{coleman}: $$|C_{\mathrm{tors}}| \leq (pg)^{4g+2}.$$
\end{theorem}

\noindent
A suitable choice of $p$ can easily be expressed in terms of the discriminant of $K$ and the conductor of $J$.

\subsection{Uniform bounds in greater dimension}
Let now $A$ be an abelian variety of dimension $g$, defined over a number field $K$ and equipped with an ample line bundle $\mathcal{L}$, so that we can define the degree of a subvariety $V$ of $A$. The number of maximal torsion cosets associated to $V$ can be bounded in terms of $V$ and $A$. In fact, Bombieri and Zannier \cite{Bombieri-Zannier96} showed that an {\it uniform} estimate can be found, where the dependence on $V$ is reduced to its geometric degree.

Later, Hrushovski gave a new proof of the Manin-Mumford conjecture through model theory, which yielded an explicit uniform bound.
\begin{theorem}[Hrushovski, \cite{Hrushovski01}]
The number $T$ of maximal torsion cosets of a subvariety $V$ of $A$ can be bounded as follows:
$$T \ll_A \deg(V)^{\alpha(A)},$$
where $\alpha(A)>0$.
\end{theorem}
    
\noindent
{\it Remarks.} The symbol $\ll_A$ means that the stated inequality is true after possibly multiplying the right member by a positive real number depending on $A$. The number $\alpha(A)$ is doubly exponential in $g$, and it also depends on a prime of good reduction for $A$.    
  
\medskip

A breakthrough came with the work of Amoroso and Viada on the effective Bogomolov conjecture for subvarieties of an algebraic torus.  They observed that it was more relevant in this setting to describe the geometry of $V$  with another parameter. Let $S$ be a semi-abelian variety; since it is quasi-projective, we can define the degree of any subvariety of $S$.
\begin{definition}\label{def:delta}
If $V$ is a subvariety of $S$, let $\delta(V)$ be the smallest $d $ such that $V$ is the intersection of hypersurfaces of $S$ which have degree at most $d$. 
\end{definition}
 
\noindent 
Let $V$ be a subvariety of $S$ and $V_{\mathrm{tors}}:= V(\bar{K}) \cap S_{\mathrm{tors}}$ be its torsion subset. It is possible to bound the degree of the $j$-equidimensional part $V_{\mathrm{tors}}^j$ of the Zariski closure $\overline{V_{\mathrm{tors}}}$ of $V_{\mathrm{tors}}$, for an integer $0 \leq j \leq \mathrm{dim}(V)$.

In the case of tori, Amoroso and Viada's theorem concerns the Zariski density of points of small height, but it has the following consequence.
\begin{theorem}[Amoroso-Viada, \cite{Amoroso-Viada09}]\label{Amoroso-Viada}
Let $V \subset \mathbb{G}_m^n$ and $0 \leq j \leq \mathrm{dim}(V)$. For any $\varepsilon >0$: $$\deg(V_{\mathrm{tors}}^j) \ll_{n, \varepsilon} \delta(V)^{n-j+ \varepsilon}.$$
\end{theorem}

\noindent
Their result was later improved by the second author. Mixing their strategy with ideas of Beukers and Smyth \cite{Beukers-Smyth02} more suited to the study of torsion points, one may find a bound with optimal dependance on $\delta(V)$.
\begin{theorem}[Mart\'inez, \cite{Martinez17}]
If $V \subset \mathbb{G}_m^n$ and $0 \leq j \leq \mathrm{dim}(V)$, then
\[
\deg(V_{\mathrm{tors}}^j) \leq 2^{6n^{3}}  \delta(V)^{n-j}.
\]
\end{theorem}

\noindent
A straightforward consequence of this theorem is an estimate on the number of maximal cosets. A further study provides a variant of this bound which proves conjectures of Aliev and Smyth~\cite{Aliev-Smyth:2012}, and Ruppert~\cite{Ruppert:1993}.

\subsection{New bounds for the torsion on subvarieties of abelian varieties} Our main theorem is an estimate of the same strength for subvarieties of an abelian variety $A$ of dimension $g \geq 1$ defined over a number field (with fixed embedding in projective space).

Combining algebraic interpolation with a theorem of Serre on homotheties in the Galois representation associated to the torsion points of $A$, we prove the following bound.
\begin{theorem}
\label{main theorem}
If $V$ is a subvariety of $A$ and $0 \leq j \leq \mathrm{dim}(V)$, then
\[
\deg(V_{\mathrm{tors}}^j) \ll_A  \delta(V)^{g-j}.
\]
\end{theorem}

\noindent
This immediately translates into a bound for the number of maximal torsion cosets. In the equidimensional case, this only depends on the degree of $V$.
\begin{corollary}\label{main corollary}
The number $T$ of maximal torsion cosets of a subvariety $V$ of $A$ is bounded as follows:
$$T \ll_A \deg(V)^{g}.$$
\end{corollary}
 
\noindent
In the case of curves, we get an improved bound.
\begin{theorem}
\label{theorem curves}
Let $C$ be a curve in $A$ which is not a torsion coset. Then
$$|C_{\mathrm{tors}}| \ll_A \deg(C)^2.$$
\end{theorem}

\noindent
The dependence on $A$ in these three estimates will be explicited below in terms of the constant that appears in Serre's theorem (which is still rather mysterious). 

\medskip

\noindent
{\it Remark.} Some results on the effective Bogomolov problem (\cite{Galateau10} under a conjecture of Serre on the ordinary primes of $A$, or \cite{Galateau12} for $V$ a hypersurface) may be combined with Amoroso and Viada's method to yield explicit bounds which are polynomial in~$\delta(V)$ but weaker than Theorem~\ref{main theorem} (or even an abelian analogue of Theorem~\ref{Amoroso-Viada}). It is not suprising since this approach does not fully exploit the properties of torsion points.

\medskip

The article is organized as follows. In the next section, we discuss and study alternate measures for the degree of subvarieties of $A$, introduce Hilbert functions and recall the classical upper bound (resp. lower bound) proved by Chardin (res. Chardin and Philippon).

In the third section, we state Serre's theorem, which is a first step towards a famous conjecture of Lang on homotheties in the Galois group of the extensions generated by torsion points of $A$. We use it several times to locate the torsion subset $V_{\mathrm{tors}}$ of an irreducible subvariety $V$ of $A$ that is not a torsion coset. We thus show that $V_{\mathrm{tors}} \subset V'$, where $V'$ is an algebraic set that satisfies some important properties and can be described precisely in terms of~$V$. We then prove Theorem \ref{theorem curves}, where algebraic interpolation is not needed and a simple application of B\'ezout's theorem is sufficient to conclude.

In the last section, our estimates on Hilbert functions allow us to interpolate $V'$ by a hypersurface of $A$ retaining most of the crucial information contained in~$V'$. We give a proof of Theorem \ref{main theorem}, and we will finally discuss its optimality in terms of~$\delta(V)$.

\medskip

\subsection*{Conventions}
Unless stated otherwise, we fix throughout this paper an abelian variety $A$ of dimension $g \geq 1$ defined over a number field~$K$.
We also fix an ample line bundle $\mathcal{L}$ on~$A$. After possibly replacing $\mathcal{L}$ by $\mathcal{L}^{\otimes 3}$, we will assume that $\mathcal{L}$ is very ample and defines a normal embedding into some projective space~$\PP^n$. In addition, after possibly tensorizing $\mathcal{L}$ by $\mathcal{L}^{-1}$, we may assume that $\mathcal{L}$ is symmetric. By abuse of language, we will say in this article that a real number depends on~$A$ when it depends on both $A$ and~$\mathcal{L}$.

A projective embedding being fixed, we may now identify every subvariety  $V$ of $A$ (not necessarily irreducible or equidimensional) with its image in~$\PP^n$. The field of definition of $V$ will be denoted by $K_V$. We will say that $V$ is {\it non-torsion} if it is not a torsion coset.

We also let $A_{\mathrm{tors}}$ be the torsion subgroup of $A$ over $\bar{K}$, and $K_{\mathrm{tors}}$ the field generated over $K$ by $A_{\mathrm{tors}}$. If $l$ is a positive integer, the $l$-torsion subgroup of $A$ will be denoted by $A[l]$, and its field of definition by $K_l$.

\section{Geometric preliminaries}

 Our approach relies strongly on fine interpolation results which follow from estimates on the Hilbert function proved by Chardin and Philippon. Before stating them at the end of this section, we will need to recall some basic geometric properties of abelian varieties, and then introduce various measures of the geometric degree for a subvariety of $A$, that naturally appear in our bounds on Hilbert functions.

\subsection{Classical facts on abelian varieties} 

We gather here classical properties about the geometry of abelian varieties, morphisms and stabilizers, which will be used frequently in the sequel. Let us start with a precious information concerning the translations in $A$.
\begin{lemma}[Lange-Ruppert, \cite{Lange-Ruppert85}]\label{lem:Lange}
The translations in $A$ can be defined by homogeneous polynomials in $K[X_0,\ldots,X_n]$ with degree at most~$2$.
\end{lemma}

\noindent
Notice that a projectively normal embedding in~$\PP^n$ is needed here. Without this assumption, the degree of the homogeneous polynomials can not be so explicitly bounded. 

The degree of a subvariety of $A$ is invariant under translation. We now describe how it behaves under some isogenies. Let $V$ be an irreducible subvariety of $A$, and for a non-zero integer $k$, denote by $[k]$ the isogeny: $A \rightarrow A$. By Lemme~6 of \cite{Hindry88},  we have 

\begin{equation}\label{eq:degree}
\deg([k]^{-1}V)=k^{2\codim(V)}\deg(V).
\end{equation}

\medskip

Suppose now that $V$ is non-torsion. We will exploit this assumption by looking at the {\it stabilizer} of $V$. 

\begin{definition}
The stabilizer of $V$ is the algebraic subgroup of $A$ defined by
$$\Stab(V):=\lbrace P \in A(\bar{K}) \;\vert\; P+V=V \rbrace.$$
\end{definition}

\noindent
Because we assume that $V$ is non-torsion, it follows from B\'ezout's theorem (\cite{Hindry88}, Lemme 8) that $$\mathrm{dim} (\Stab(V)) < \mathrm{dim}(V),$$ and $$\deg (\Stab(V)) \leq 2^g \deg (V)^{\dim(V)+1}.$$ By Poincar\'e's complete reducibility theorem, the abelian variety $A$ is isogenous to a product $$B \times \Stab(V)^0,$$ where $\Stab(V)^0$ is the connected component of $\Stab(V)$ which contains~$0$, and $B$ is an abelian subvariety of $A$ (\cite{Hindry88}, Lemme 9). After composing with an isogeny whose kernel is $\Stab(V)/\Stab(V)^0$, we find a surjective homomorphism
\begin{equation}\label{eq:homomorphism-stabilizer}
\varphi_V \colon A \rightarrow B,
\end{equation}
with $\ker \varphi_V = \Stab(V)$. Taking $K$ large enough so that all the simple factors of $A$ are defined over $K$, we may assume that $\varphi_V$ is defined over $K$.

\subsection{Degrees of definition and Hilbert functions}

A key point in our approach is to use some (classical) refined variants of the projective degree.
If $V$ is a subvariety of $A$, we define its degree to be the sum of the degrees of its irreducible components. We have already introduced $\delta(V)$ as the minimal degree of hypersurfaces of $A$ with intersection $V$. The next definition only retains the projective nature of $V$. 

\begin{definition} The degree of incomplete (resp. complete) definition of~$V$, denoted by $\delta_0(V)$ (resp. $\delta_1(V)$), is the minimal $d$ such that the irreducible components of $V$ are irreducible components of an intersection (resp. $V$ is an intersection) of hypersurfaces in $\PP^n$ which all have degree at most~$d$.
\end{definition}

\noindent
For a family of subvarieties $V_1,\ldots,V_t$ of $A$, we easily get the following inequality, see for instance (\cite{Martinez17}, Lemma~2.6):
\begin{equation}\label{eq:delta}
\delta_1 \Big(\bigcup_{i=1}^t V_i\Big)\leq \sum_{i=1}^t\delta_1(V_i).
\end{equation}

\noindent
We have the following inequalities between our different degrees.
\begin{lemma}\label{delta-ineq}
If $V\subset A$ is an equidimensional variety, we have
\begin{center}
\begin{enumerate}[label=(\roman*)]
\item $\delta_0(V)\leq \delta_1(V) \leq \deg(V)$,
\smallskip
\item $\delta_1(V) \leq \delta(V) \leq \deg(A) \delta_1(V)$.
\end{enumerate}
\end{center}
\end{lemma}
\begin{proof}
The first inequality is straightforward. The image of $V$ by a linear map $\PP^n \rightarrow \PP^{\dim(V)+1}$ has degree at most $\deg(V)$, and the variety $V$ is the intersection of hypersurfaces of $\PP^n$ obtained by pull-backs of such linear maps. This shows that $$\delta_1(V) \leq \deg(V).$$

Now, if $Z$ is a hypersurface of $A$, we have $\delta_1(Z) \leq \deg(Z)$. We choose a set of hypersurfaces $Z_i$ of $A$ of degree at most $\delta(V)$ and such that $V= \bigcap_i Z_i$. We get:
$$\delta_1(V) \leq \max_i \delta_1(Z_i) \leq \max_i \deg(Z_i) \leq \delta(V).$$
Assume finally that $V= \bigcap_i Z_i$ where the $Z_i$'s are hypersurfaces of $\PP^n$. After possibly removing some of the $Z_i$'s, we have $V= \bigcap_i Z_i \cap A$, where $Z_i \cap A$ is a hypersurface of $A$. The last inequality is then a direct consequence of B\'ezout's theorem. 
\end{proof}




The degrees of complete and incomplete definition do not necessarily behave as the usual degree with respect to translations in~$A$. However, we have the following useful comparison.
\begin{lemma}\label{lem:deltalemma}
Let $V$ be a subvariety of $A$.
If $P\in A$, then
\[
\delta_1(P+V)\leq 2\,\delta_1(V)
\quad\mbox{and}\quad
\delta_0(P+V)\leq 2\,\delta_0(V).
\]
\end{lemma}

\begin{proof}
By Lemma~\ref{lem:Lange}, if we have a set of complete (resp. incomplete) equations of degree $d >0$ for $V$, we get a set of complete (resp. incomplete) equations of degree $2d$ for $V+P$. This proves the announced inequalities.
\end{proof}

\medskip

We now introduce the Hilbert function that can be attached to any projective variety. The incomplete degree of definition naturally arises in a classical lower bound on this function, which explains why we needed to introduce and study this degree.

To a subvariety $V \subset A \subset \PP^n$, there corresponds the homogeneous ideal $I$ of polynomials of $\bar{K}[X_0, \ldots, X_n]$ which vanish on $V$.
This defines a graded $\bar{K}[X_0, \ldots, X_n]$-module $\bar{K}[X_0, \ldots, X_n]/I$. For $\nu$ a positive integer, we let
\[
H(V;\nu):=\dim(\bar{K}[X_0, \ldots, X_n]/I)_{\nu}
\]
the Hilbert function of $V$ at $\nu$. We start with a classical upper bound on the Hilbert function.
\begin{theorem}[Chardin, \cite{Chardin89}]\label{thm:HilbertUpperBound}
If $V$ is an equidimensional variety of dimension~$d$ and $\nu\in\mathbb{N}$, then
\[
H(V;\nu)\leq 
\begin{pmatrix}
\nu + d\\ d
\end{pmatrix}
\deg(V).
\]
\end{theorem}

\noindent
On the other hand, a refined version of Chardin and Phillipon's theorem on Castelnuovo's regularity yields a lower bound for the Hilbert function when $\nu$ is large enough in a precise sense. The following is Th\'eor\`eme 6.1 of \cite{Amoroso-Viada12}   (see  \cite{ChardinPhilippon99}, Corollaire~3 for the original statement).
\begin{theorem}[Chardin-Philippon] \label{thm:HilbertLowerBound}
Let $V:= \bigcup_{1 \leq j \leq s} V_j$ an union of equidimensional varieties of dimension~$d$ and $m:= s-1+ (n-d)\sum_{j=1}^s(\delta_0(V_j)-1)$. For any integer $\nu>m$, we have
\[
H(V;\nu)\geq
\begin{pmatrix}
\nu+d-m\\ d
\end{pmatrix}
\deg(V).
\]
\end{theorem}

\noindent
{\it Remark.} Combining these two bounds provides a powerful interpolation tool, yielding for a well chosen pair of varieties $(V,V')$ a hypersurface of controlled degree which contains $V$ and avoids $V'$.

\section{Galois properties of torsion points and geometric consequences}

In this section, we use a deep theorem of Serre on Galois representations to locate the torsion subset of a subvariety of~$A$. Our strategy is primarily based on the classical approach to the Manin-Mumford conjecture initiated by Lang in~\cite{Lang65}.

\subsection{Homotheties in the image of Galois}

For every prime number~$\ell$, let $T_{\ell}(A)$ be the $\ell$-adic Tate module of~$A$. There is a representation
\[
\rho_{\ell}: G_K:=\mathrm{Gal}(\bar{K}/K) \longrightarrow \mathrm{GL}_{\mathbb{Z}_{\ell}}\big(T_{\ell}(A)\big),
\]
induced by the action of $G_K$ on the torsion subgroup of $A$. A long-standing conjecture of Lang states that the image of the absolute Galois group $G_K$ in the adelic representation
\[
\rho:= \prod_{\ell} \rho_{\ell}: G_K \longrightarrow \prod_{\ell}  \mathrm{GL}_{\mathbb{Z}_{\ell}}\big(T_{\ell}(A)\big)
\]
contains an open subgroup of the group of homotheties. In \cite{Bogomolov80}, Bogomolov proved that for a fixed prime number~$\ell$, the group $\rho_{\ell}(G_K)$ contains an open subgroup of the group of homotheties. Serre later showed that $\rho(G_K)$ contains a fixed power of every admissible homothety. We will use the following version of his theorem.
\begin{theorem}[Serre]
\label{homotheties}
There is an integer $c(A) \geq 1$ such that, for any two coprime positive integers $l$ and~$k$, there exists an automorphism $\sigma \in G_K$ satisfying  $$\sigma |_{ A[l]}= \big[k^{c(A)}\big].$$
\end{theorem}
\begin{proof}
This is  \cite{Wintenberger02}, Th\'eor\`eme~3. See also \cite{Serre86} p.136, Th\'eor\`eme~2 for the original statement, or \cite{Hindry88}, Lemme~12.
\end{proof}

\noindent
{\it Remark.} The problem of finding an explicit $c(A)$ in terms of $A$ - and of the field $K$ over which $A$ is defined - is still open and discussed in  \cite{Wintenberger02}, Section~2. In order to simplify notations, and since $A$ is fixed, we will write $c$ for $c(A)$ in the sequel. 

\medskip

Let $V$ be an irreducible non-torsion subvariety of $A$. Serre's theorem is our main tool to find a strict subvariety of $V$ that contains the torsion subset $V_{\mathrm{tors}}$ of~$V$. We let $X:= \varphi_V(V)$, where $\varphi_V$ was defined in (\ref{eq:homomorphism-stabilizer}). 

We distinguish several cases according to wether and how $K_X \subset K_{\mathrm{tors}}$, and first tackle the simpler case where $K_X$ is not contained in $K_{\mathrm{tors}}$.
\begin{lemma}\label{lemma:V' no torsion}
If $K_X \not\subset K_{\mathrm{tors}}$, there is a conjugate $V'$ under the action of $G_K$ such that
$$V_{\mathrm{tors}} \subset V \cap V' \subsetneq V.$$
\end{lemma}
\begin{proof}
The isogeny $\varphi_V$ is defined over $K$, so the extension $K_V/K_V \cap K_{\mathrm{tors}}$ is strict and there is a nontrivial field isomorphism $\sigma: K_V \rightarrow \bar{K}$ such that $$\sigma | _{K_V \cap K_{\mathrm{tors}}}= \mathrm{Id}.$$ Since $\sigma$ acts trivially on $V_{\mathrm{tors}}$, the lemma follows. 
\end{proof}

\subsection{Scanning the field of definition}

For the remaining of this section, we now assume that $K_X \subset K_{\mathrm{tors}}$. In comparison with the toric case, some technical difficulties arise here because of the complexity of $K_{\mathrm{tors}}$, and because of the gap between Lang's conjecture and Serre's theorem. 

We start with a preliminary examination of~$K_X$, and we define two integers $M$ and $N$ that quantify more precisely the link between $K_X$ and~$K_{\mathrm{tors}}$. Since $K_X$ is a number field, there is an integer $l \geq 1$ such that
\[
K_X \subset K_l.
\]
We denote by $v_2$ the $2$-adic valuation of an integer. For the remaining of the section, fix $M \geq 1$ the smallest integer such that the latter inclusion holds, and 
\[
v_2(M) \geq c_2+2,
\]
where $c_2:=v_2(c)$. For $P \in A[M]$, we consider the subset of the integers
\[
\mathcal{N}(P)= \Big\{ \alpha > -v_2(M), \exists \, \sigma\in G_K, \sigma |_{ A[M]} = \big[ (1+2^{\alpha} M)^c \big]  \mathrm{and} \; \sigma | _{K_{X+P}} = \mathrm{Id} \Big\}.
\]
This set contains $\mathbb{N}$ because $[M] |_{A[M]}=0$ and $K_{X+P} \subset K_M$.
Let $\beta(P)$ be the biggest integer not in $\mathcal{N}(P)$, and
\[
\beta:= \min_{P \in A[M]} \{ \beta(P) \}.
\]
In particular, notice that both $\beta(R)$ and $\beta$ are always negative.
Finally, we set \[ N:=2^{\beta +1}M.\]

We will use repeatedly the following computation based on the properties of binomial coefficients. 
\begin{lemma}
\label{binomials}
Let $2 \leq \gamma \leq \delta$ be two integers. If $k$ is an integer with $v_2(k) \geq 2$: $$v_2 \left((^{\delta}_{\gamma}) k^{\gamma}\right) \geq v_2(k) +v_2(\delta)+1.$$
\end{lemma}
\begin{proof}
Recall that $$(^{\delta}_{\gamma})= \frac{\delta}{\gamma}(^{\delta-1}_{\gamma-1}).$$
We deduce
\begin{eqnarray*}
v_2\left((^{\delta}_{\gamma}) k^{\gamma}\right) & \geq & v_2(\delta) - v_2(\gamma) + \gamma v_2(k) \\ 
& \geq & v_2(k \delta) + (\gamma-1)v_2(k) -v_2(\gamma) \\
& \geq & v_2(k \delta) + 2\gamma - 2 - v_2(\gamma)
\end{eqnarray*}
If $\gamma \geq 3$, then $\gamma-v_2(\gamma)\geq 0$, and therefore $$2 \gamma -2 -v_2(\gamma) \geq \gamma -2 \geq 1.$$
This inequality is still true for $\gamma=2$, so the proof of the lemma is complete.
\end{proof}

\medskip

\noindent
We will need to compare the $2$-adic valuations of $M$ and $N$.
\begin{lemma}
\label{two-adic}
We have
$$v_2(N)+c_2 \leq v_2(M).$$ 
\end{lemma}
\begin{proof}
The lemma follows directly from these inequalities: 
$$\beta(P) \leq -c_2-1,\quad \forall P\in A[M]$$ 
This is trivially true if $c_2=0$, so we can assume that $c_2 \geq 1$. If $\alpha \geq - c_2$, we compute $$(1+2^{\alpha}M)^c=1+c2^{\alpha}M + \sum_{2 \leq \gamma \leq c}(^c_{\gamma})(2^{\alpha}M)^{\gamma}.$$
By choice of $M$, $v_2(M)\geq c_2+2$. Hence,  $v_2(2^{\alpha} M) \geq 2$, and Lemma \ref{binomials} shows that
$$\forall \gamma\geq 2,\quad v_2\left((^c_{\gamma}) (2^{\alpha}M)^{\gamma}\right) \geq v_2(M)+ \alpha + c_2 + 1 > v_2(M) .$$
Therefore $$[(1+2^{\alpha}M)^c] |_{A[M]}=\mathrm{Id}.$$
Now, for all $P \in A[M]$, the variety $X+P$ is defined over $K_M$ so $\alpha \in \mathcal{N}(P)$ and the stated inequality holds.
\end{proof}

\noindent
{\it Remark.} With $M$ being fixed, we can associate in exactly the same way an integer $\beta_R$ to each translate of $X$ by an $M$-torsion point~$R$. For a good choice of $R$, we get that $\beta_R=\beta_R(0)$. After possibly replacing $V$ (resp. $X$) by $V+R$ (resp. $X+R$), we will now assume that $\beta=\beta(0)$. This will have no effect on our subsequent geometric construction, because the properties of $V_{\mathrm{tors}}$ that we want to prove are invariant under translation by a torsion point.

\subsection{The torsion subset of $X$}

We are ready to locate the torsion of $X$ when $K_X \subset K_{\mathrm{tors}}$. The following proposition gives an explicit description of an algebraic subset of $X$ that contains $X_{\mathrm{tors}}$.

\begin{proposition}\label{prop:V'}
There are two automorphisms $\sigma, \rho \in \mathrm{Gal}(\bar{K}/K)$ depending only on~$M$, such that if 
\[
X':=  \bigcup_{\substack{P \in B[4c] }}  \big[2^{c}\big]^{-1} \big(X^{\sigma} + P \big)  \; \;  \cup \bigcup_{P \in {B[2]\setminus \{0\}}} \big( X+P \big)  \; \;  \cup \bigcup_{P \in {B[2]}} \big( X^{\rho}+P \big),\]
then $$X_{\mathrm{tors}} \subset X \cap X'.$$
\end{proposition}
\begin{proof}
Fix $Q \in X_{\mathrm{tors}}$ with exact order $L \geq 1$, and let
\[
m:=\lcm (L,M),\quad\mbox{and}\quad u:= \prod_{p \neq 2}p^{\max \{ 0, v_p(L)-v_p(M)\}}.
\]
Since $u$ is odd, B\'ezout's identity yields an odd positive integer $v$ and an integer $w$ such that $$2^{v_2(M)}w + uv=1.$$ We consider different cases according to $v_2(L)$.

\medskip

\noindent
{\it Case 1.} Assume on the one hand that $v_2(L) \leq c_2 + 2$. Let
\[
l:= \prod_{p \neq 2}p^{\max \{v_p(L),v_p(M)\}}v=2^{-v_2(M)}Muv.
\]
Since the integers $m$ and $2+l$ are coprime, we are in a position to use Theorem~\ref{homotheties}.
This gives an automorphism $\sigma \in G_K$ such that \[\sigma |_{A[m]}=\big[ (2+l)^{c} \big].\] 
Looking at the action of $\sigma$ on $Q$, we obtain:
\[Q^{\sigma}= \big[2^{c}\big] Q + \sum_{1 \leq \gamma \leq c} \big[(^c_{\gamma})2^{c-\gamma} l^{\gamma} \big] Q = \big[2^{c}\big] Q - P,\]
where $P\in\Ators$ has order dividing $ 2^{v_2(L)}$.
In particular $P\in A[4c]$, since $$2^{v_2(L)}\mid 2^{c_2+2} \mid 4c.$$
We immediately get \[Q \in  \big[2^{c}\big]^{-1} \big(X^{\sigma} + P \big) \subset X'.\]
By construction, we also have \[\sigma |_{A[M]}=\big[ (2+2^{\alpha}M)^{c} \big],\] with $\alpha=-v_2(M)$, so the action of $\sigma$ on $X$ does not depend on $Q$.

\medskip

\noindent
{\it Case 2.} Assume on the other hand that $v_2(L) \geq c_2 + 3$.
We first examine the case where $v_2(L) > c_2+v_2(N)$.
Let  
\[
l' = \frac{2^{v_2(L)}}{2^{c_2+1}}\prod_{p \neq 2} p^{\max \{v_p(L), v_p(M) \}} v=\frac{2^{v_2(L)}}{2^{c_2+1}} 2^{-v_2(M)}Muv.
\]
Remark that the prime divisors of $m$ also divide $l'$. Therefore $m$ and $l'+1$ are coprime and Theorem~\ref{homotheties} gives an automorpshim $\tau \in G_K$ such that
\[
\tau |_{A[m]}=\big[ (1+l')^{c} \big].
\]
We check that $v_2(l') = v_2(L)-c_2-1 \geq 2$, so Lemma~\ref{binomials} shows that, for any integer $2 \leq \gamma \leq c$,  
\[
v_2\left((^c_{\gamma}) l'^{\gamma}\right) \geq v_2(l') + c_2 + 1 = v_2(L).
\]
Hence, the action of $\tau$ on~$Q$ yields
\[
Q^{\tau}= Q + [c l']Q=Q-P,
\]
 where $P\in\Ators$ has exact order~$2$. So, we find that $Q \in X^{\tau}+ P$.
Furthermore, we have that \[\tau |_{A[M]}=\big[ (1+2^{\alpha}M)^{c} \big],\]
with $$\alpha:=-v_2(M)+v_2(L)-c_2-1 \geq v_2(N)-v_2(M) \geq \beta+1.$$
So $\alpha \in \mathcal{N}(0)$ and we have $\tau |_{K_X}= \mathrm{Id}$, by definition of~$\mathcal{N}(0)$.
We derive $X^{\tau}=X$, and $Q \in X+P\subset X'$.

\medskip

\noindent
{\it Case 3.} Assume finally that $c_2+3 \leq v_2(L) \leq c_2+v_2(N)$. Let 
\[
l''= \frac{2^{c_2+v_2(N)}}{2^{c_2+1}} \prod_{p \neq 2} p^{\max \{v_p(L), v_p(M) \}}v =  2^{v_2(N)-1-v_2(M)} Muv,
\]
which is an integer by the assumption.
Again, $m$ and $1+l''$ are coprime, and Theorem~\ref{homotheties} ensures that there exists an automorphism $\rho \in G_K$ such that 
\[
\rho |_{A[m]}=\big[ (1+l'')^{c} \big].
\]
Notice that $v_2(l'') =v_2(N)-1 \geq 2$, so by Lemma \ref{binomials}, for any integer $2 \leq \gamma \leq c$,
\[
v_2\left((^c_{\gamma}) l''^{\gamma}\right) \geq v_2(l'') + c_2 + 1 \geq v_2(N) + c_2 \geq v_2(L).
\]
From this, we derive $$Q^{\rho}= Q+ [c l'']Q=Q-P,$$ where $P\in A[2]$.
Hence \[Q \in X^{\rho} +P\subset X'.\] 
We check that the action of $\rho$ on $X$ does not depend on $Q$, since $$\rho |_{A[M]}=\big[ (1+2^{\alpha}M)^{c} \big],$$ with $\alpha=\beta=v_2(N)-1-v_2(M)\geq 2-v_2(M).$
\end{proof}

\medskip

\noindent
{\it Remark.} The proof of the proposition shows that we can choose $\sigma$ and $\rho$ such that
\[
\sigma |_{A[M]}= \big[(2+2^{-v_2(M)}M)^c \big] \quad \mathrm{and} \quad \rho |_{A[M]}= \big[(1+ 2^{\beta}M)^c \big],
\]
where the components of $X'$ concerning $\rho$ only appear when $\beta\geq 2-v_2(M)$. 
This completely describes the action of $\sigma$ and $\rho$ on $X$.

\subsection{Pulling back from $X$ to $V$.} At this point, the important condition that $X$ does not lie in the algebraic set $X'$ is still missing in our construction. This will be fixed by pulling back from $X$ to $V$ and exploiting the classical properties of the stabilizer.
Let $V'\subset A$ be the preimage of $X'$ by~$\varphi_V$, i.e.
\[
V':=  \bigcup_{\substack{P \in \varphi_V^{-1}(B[4c]) }}  \big[2^{c}\big]^{-1} \big(V^{\sigma} + P \big)  \; \;  \cup \bigcup_{P \in {\varphi_V^{-1}(B[2]\setminus \{0\})}} \big( V+P \big)  \; \;  \cup \bigcup_{P \in {\varphi_V^{-1}(B[2])}} \big( V^{\rho}+P \big),
\]
for $\sigma,\rho\in\Gal(\bar{K}/K)$ chosen as in Proposition~\ref{prop:V'} (and the remark below the proof of this proposition).

\begin{lemma}\label{lemma:V'-stabilizer}
We have
\[
V_{\mathrm{tors}} \subset V \cap V' \subsetneq V.
\]
\end{lemma}
\begin{proof}
First remark that $$\varphi_V(V_{\mathrm{tors}}) \subset X_{\mathrm{tors}} \subset X',$$ so $V_{\mathrm{tors}} \subset \varphi_V^{-1}(X') \subset V'$ because $\varphi_V$ is an isogeny defined over $K$. Thus $$\varphi_V^{-1} \varphi_V(V)=V + \ker \varphi_V=V + \mathrm{Stab}(V)=V.$$
We now check that the inclusion $V \cap V' \subset V$ is strict, and we cut the proof in three pieces corresponding to each type of component of $V'$. 

\medskip

\noindent
{\it Case 1.} Suppose first that there is $P \in \varphi_V^{-1}(B[4c])$ such that $$V \subset \big[2^{c}\big]^{-1} \big(V^{\sigma} + P \big).$$ We derive $$\bigcup_{R \in [2^c]^{-1} \mathrm{Stab}(V)} \big(V + R \big) \subset \big[2^{c}\big]^{-1} \big(V^{\sigma} + P \big).$$
Using for instance \cite{Hindry88}, Lemme 6, we compare the degrees of these two algebraic sets and find $$2^{2c \cdot \mathrm{codim} (\mathrm{Stab}(V))} \deg(V) \leq 2^{2c \cdot \mathrm{codim}(V)} \deg(V).$$ This yields $$\dim(V) \leq \dim (\mathrm{Stab}(V)),$$ which is a contradiction since $V$ is not a torsion subvariety of $A$. 

\medskip

\noindent
{\it Case 2.} This is where we really use the properties of our isogeny $\varphi_V$. Suppose that $P \in \varphi_V^{-1}(B[2] \setminus \{ 0\})$. Then we have that $P \notin \mathrm{Stab}(V)$, and so $$V+P \neq V.$$

\medskip

\noindent
{\it Case 3.} After composing by $\varphi_V$, we are reduced to considering the possibility that there exist $P \in A[2]$ such that $X^{\rho}+P=X$. We can find a torsion point $R \in A[2^{c_2+v_2(N)}]$ such that $[c2^{\beta}M]R=P$. We may also assume that $\beta \geq 2-v_2(M)$ (see the remark following the proof of Proposition~\ref{prop:V'}). Using Lemma~\ref{binomials}, we see that 
\begin{eqnarray*}
(X+R)^{\rho} & = & X^{\rho} + \big[(1+2^{\beta}M)^c\big]R \\
& = & X -P + R + P \\
& = & X + R.
\end{eqnarray*}
So $X+R$ is fixed by $\rho$ and  $\beta \in \mathcal{N}(R)$. Furthermore, if $\alpha \geq \beta +1$, there is $\tau \in G_K$ such that $\tau |_{A[M]}=[(1+2^{\alpha}M)^c]$ and $\tau |_{K_X}=\mathrm{Id}$. We compute $$(1+2^{\alpha}M)^c= 1+c2^{\alpha}M + \sum_{2 \leq \gamma \leq c}(^c_{\gamma})(2^{\alpha}M)^{\gamma}.$$ Since $v_2(2^{\alpha}M) \geq 3$, we can apply Lemma~\ref{binomials} again to obtain, for $\gamma \geq 2$,
\[
v_2\left((^c_{\gamma}) (2^{\alpha}M)^{\gamma}\right)  \geq v_2(N) + c_2 + 1.
\]
This means that $\tau(R)=R$, and so $$(X+R)^{\tau}= X^{\tau}+R^{\tau}=X+R.$$
Thus $\alpha \in \mathcal{N}(R)$, and we finally have $$\beta \leq \beta(R) \leq \beta-1,$$ which yields a contradiction.
\end{proof}

\subsection{The case of curves}

The information contained in Lemma~\ref{lemma:V' no torsion} and Lemma~\ref{lemma:V'-stabilizer} is enough to bound the number of torsion points
on a curve $C\subset A$. The following is a precise version of Theorem~\ref{theorem curves}.

\begin{proposition}
Let $C\subset A$ be an irreducible non-torsion curve.
Then
\[
\vert C_{\mathrm{tors}}\vert \leq 4^{(2c+1)g}\deg(C)^2.
\]
\end{proposition}
\begin{proof}
We notice that since $C$ is non-torsion, its stabilizer is finite. Let $X:=\phi_C(C)$ and assume that $K_X \not\subset K_{\mathrm{tors}}$. By Lemma~\ref{lemma:V' no torsion}, there is a conjugate $C^{\sigma}$ of $C$ such that $$C_{\mathrm{tors}}\subset C\cap C^{\sigma} \subsetneq C.$$ We are thus in a position to use B\'ezout's theorem: $$|C_{\mathrm{tors}}| \leq \deg(C) \cdot \deg(C^{\sigma}) = \deg(C)^2,$$
and this bound is stronger than that stated in the proposition.

Now, if $K_X \subset K_{\mathrm{tors}}$, Lemma~\ref{lemma:V'-stabilizer} shows that $C_{\mathrm{tors}}\subset C\cap C'\subsetneq C$, where
\[
C'=\bigcup_{P\in \phi_C^{-1}(B[4c])}[2^c]^{-1}(C^{\sigma}+P)\quad \cup\bigcup_{P\in \phi_C^{-1}(B[2] \setminus \{0\})}(C+P)\quad\cup\bigcup_{P\in \phi_C^{-1}(B[2])}(C^{\rho}+P),
\]
for some $\sigma,\rho\in\Gal(\bar{K}/K)$.
We use ~\eqref{eq:degree} and the invariance of the degree under translation, then we add up the degrees of all the components of $C'$ to obtain
\begin{eqnarray*}
\deg(C') & \leq & \big( (4c)^{2g}2^{2c(g-1)}\,+\, (2^{2g}-1)\,+\,2^{2g}\big)\deg(C) \\
& \leq & \big( 2^{4g+2cg-2c}c^{2g}+2^{2g+1}-1\big)\deg(C) \\
& \leq & \big( 2^{2g(c+2)-2c}c^{2g} +2^{2g+1} \big)\deg(C) \\
& \leq & \big(2^{(2c+1)2g-2c} + 2^{2g+1} \big) \deg(C) \\
& \leq & 2^{(2c+1)2g-2c+1} \deg(C) \\
& \leq & 4^{(2c+1)g} \deg(C).
\end{eqnarray*}
The statement follows once again by B\'ezout.
\end{proof}

\noindent
{\it Remarks.} If $C$ is not defined over $K_{\mathrm{tors}}$, our proof gives a much stronger bound:
\[
\vert C_{\mathrm{tors}}\vert\leq \deg(C)^2.
\]
If we consider the case where $A$ is the Jacobian $J$ of $C$, both $C$ and $J$ are defined over the same number field (in particular, the third union in the definition of $V'$ disappears). If we consider the canonical embedding of the Jacobian, we have $\deg(C)=g$ and a quick computation gives
\[
\vert C_{\mathrm{tors}}\vert \leq  4^{(2c+2)g}.
\]

\section{Bounding the torsion through interpolation}

We turn to the proof of our main theorem for general subvarieties of  $A$. Since a simple iterated application of B\'ezout's theorem like in the case of curves would yield a bound far weaker than expected, we will follow a strategy based on the existence of a nice obstructing hypersurface through refined interpolation tools. 

\subsection{The interpolation machine}

We first build a preliminary interpolation machine suited to our situation. This is mainly derived from Chardin and Philippon's estimates for Hilbert functions.
\begin{lemma}\label{lemma:interpolation}
Let $V$ be an irreducible subvariety of $A$,
as well as $\sigma \in\Gal(\overline{K}/K)$, $P\in\Ators$ and $k\geq 2$ an integer.
\begin{enumerate}[label=(\roman*)]
\item\label{item:interpolation1}
If $P+V^{\sigma}\neq V$, there exists a hypersurface $Z$ of $\PP^n$ such that $P+V^{\sigma}\subset Z$, $V \not\subset Z$ and $\deg(Z) \leq 6ng \, \delta_0(V)$.
\item\label{item:interpolation2}
If $V$ is non-torsion, there exists a hypersurface $Z'$ of $\PP^n$ such that $[k]^{-1}(P+V^{\sigma})\subset Z'$, $V\not\subset Z'$ and $\deg(Z') \leq 4ng \, k^{2g} \, \delta_0(V)$.
\end{enumerate}
\end{lemma}

\noindent
{\it Remark.} The conclusion of (ii) implies that  $V\not\subset [k]^{-1}(P+V^{\sigma})$. In fact, this follows directly from the fact that $V$ is non-torsion, using the same argument as above in the proof of Lemma \ref{lemma:V'-stabilizer}, Case 1. 

\begin{proof}
The arguments of proof are similar in each case. We start with~\ref{item:interpolation1}. Notice that $P+V^{\sigma}$ is also an irreducible subvariety of $A$.
By Theorem~\ref{thm:HilbertUpperBound}, for any positive integer $\nu$
\[
H(P+V^{\sigma};\nu)\leq
\begin{pmatrix}
\nu + d\\ d
\end{pmatrix}
\deg(V),
\]
where $d= \dim(V)$. Let $\widetilde{V}:=V\cup P+V^{\sigma}$. This is an equidimensional variety of dimension~$d$ and degree~$2\deg(V)$.
By Theorem~\ref{thm:HilbertLowerBound}, for any $\nu>m$
\[
H(\widetilde{V};\nu)\geq
\begin{pmatrix}
\nu + d - m \\ d
\end{pmatrix}
2\deg(V),
\]
with $$m:=1+ (n-d)(\delta_0(V)+ \delta_0(P+ V^{\sigma})-2).$$
Using both inequalities with $\nu:=m(2d+1)$, we obtain
\begin{eqnarray*}
\frac{H(P+V^{\sigma};\nu)}{H(\widetilde{V};\nu)}
& \leq & \frac{1}{2} \begin{pmatrix}
\nu + d \\ d
\end{pmatrix} \begin{pmatrix}
\nu + d - m \\ d
\end{pmatrix}^{-1}
 \leq \frac{1}{2} \left(1+\frac{m}{\nu-m} \right)^d \\
& \leq & \frac{1}{2} \left( 1+\frac{1}{2d} \right)^d \leq \frac{\sqrt{\mathrm{e}}}{2} <1.
\end{eqnarray*}
Hence, there is a hypersurface~$Z$ of $\PP^n$ of degree $\nu$ such that $P+V^{\sigma}\subset Z$ and $\widetilde{V}\not\subset Z$. The last inclusion implies that $V \not\subset Z$. Moreover, Lemma~\ref{lem:deltalemma} gives $\delta_0(P+V^{\sigma})\leq 2\delta_0(V^{\sigma})=2\delta_0(V)$, so we obtain:
\[
\deg(Z)\leq 3(2d+1)(n-d)\delta_0(V) \leq 6ng \delta_0(V),
\]
concluding the proof of~\ref{item:interpolation1}.

\medskip

We now turn our attention to assertion~\ref{item:interpolation2}.
To simplify notations, we let $$W:=[k]^{-1}(P+V^{\sigma}).$$
It is an equidimensional subvariety of $A$ of dimension~$d$, and by~\eqref{eq:degree}, we have $\deg(W)=k^{2(g-d)}\deg(V)$.
Theorem~\ref{thm:HilbertUpperBound} gives, for any positive integer $\nu$,
\[
H(W;v)\leq
\begin{pmatrix}
\nu + d\\ d
\end{pmatrix}
k^{2(g-d)}\deg(V).
\]
Consider $$\widetilde{W}:= \bigcup_{Q\in[k]^{-1}\Stab(V^{\sigma})} Q+V.$$
Recall here that $\varphi_V$ given by  (\ref{eq:homomorphism-stabilizer}) is assumed to be defined over $K$, so that $$\Stab(V^{\sigma})= \Stab(V).$$ Since $V$ is non-torsion, the variety $\widetilde{W}$ is equidimensional  of dimension $d$ and degree~$k^{2r}\deg(V)$, where we denote $r:= \codim(\Stab(V))$. 
By Theorem~\ref{thm:HilbertLowerBound}, for any $\nu>m$
\[
H(\widetilde{W};\nu)\geq
\begin{pmatrix}
\nu + d - m \\ d
\end{pmatrix}
k^{2r}\deg(V),
\]
where $$m=1+ (n-d) \sum_{Q \in [k]^{-1}\Stab(V^{\sigma})}(\delta_0(V+Q)-1) \leq 2(n-d)q^{2g} \delta_0(V)$$
and the last inequality follows from Lemma~\ref{lem:deltalemma}. We know that $r>g-d$ since $V$ is non-torsion, and $$k^{2(g-d)-2r}\leq k^{-2}<\mathrm{e}^{-1}.$$
Fixing $\nu=m(d+1)$, we obtain the following inequality:
\[
\frac{H(W;\nu)}{H(\widetilde{W};\nu)}
\leq
k^{2(g-d)-2r}
\begin{pmatrix}
	\nu+d\\ d
\end{pmatrix}
\begin{pmatrix}
	\nu+d-m\\ d
\end{pmatrix}^{-1}
<
\mathrm{e}^{-1} \left( 1+\frac{1}{d} \right)^d
<1.
\]

Thus, there is a hypersurface $Z'_0$ of $\PP^n$ of degree $\nu$ such that $W\subset Z'_0$ and $\widetilde{W} \not\subset Z'_0$. This means that there exists $$R \in [k]^{-1}\Stab(V^{\sigma})$$ such that $R+V \not\subset Z_0'$. If we let $Z':=-R+ Z'_0$, we see that $V \not\subset Z'$. On the other hand, we immediately check that $R+W \subset W$,
so $W\subset Z'$. By Lemma~\ref{lem:Lange}, we  obtain
\[
\deg(Z')\leq 2 \nu \leq 4ng\,k^{2g} \delta_0(V),
\]
which ends the proof of~\ref{item:interpolation2}.
\end{proof}

\subsection{The obstructing hypersurface}

The interpolation machine applied to our Lemma \ref{lemma:V'-stabilizer} yields an obstructing hypersurface for~$V$. This hypersurface contains~$V_{\mathrm{tors}}$ and its degree is precisely controlled.
\begin{proposition}\label{prop:Z}
If $V$ is an irreducible non-torsion subvariety of $A$,
there exists a hypersurface $Z\subset\PP^n$ such that $V_{\mathrm{tors}} \subset V\cap Z\subsetneq V$ and
$$\deg(Z) \leq 16^{g(c+1)}\,n\,\delta_0(V).$$
\end{proposition}
\begin{proof}
We are going to apply Lemma~\ref{lemma:interpolation} to each of the components that appear in the
various definitions of $V'$ above. 

Assume first that $K_X \not\subset K_{\mathrm{tors}}$. We apply Lemma~\ref{lemma:V' no torsion} to $V$ and obtain a variety $V'$ which is a conjugate of $V$ under the action of $\Gal(\bar{K}/K)$, such that $$V_{\mathrm{tors}} \subset V\cap V'\subsetneq V.$$
By Lemma~\ref{lemma:interpolation}\ref{item:interpolation1}, we get a hypersurface $Z$ in $\PP^n$ such that $V\cap V'\subset V\cap Z\subsetneq V$ and $$\deg(Z) \leq 6ng \, \delta_0(V),$$ and we rapidly check that this is stronger than the announced bound.

If $K_X \subset K_{\mathrm{tors}}$, we apply Lemma~\ref{lemma:V'-stabilizer}
to $V$ and obtain that
$V_{\mathrm{tors}}\subset V\cap V'\subsetneq V$, where
\[
V'=\bigcup_{\substack{P \in \varphi_V^{-1}(B[4c]) }}  \big[2^{c}\big]^{-1} \big(V^{\sigma} + P \big)  \; \;  \cup \bigcup_{P \in {\varphi_V^{-1}(B[2]\setminus \{0\})}} \big( V+P \big)  \; \;  \cup \bigcup_{P \in {\varphi_V^{-1}(B[2])}} \big( V^{\rho}+P \big),
\]
for some $\sigma,\rho\in\Gal(\bar{K}/K)$.

By Lemma~\ref{lemma:interpolation}\ref{item:interpolation2}, for each $P\in\varphi_V^{-1}(B[4c])$,
there exists a hypersurface $Z_{1,P}$ of degree at most $2^{2gc+2} \, ng \, \delta_0(V)$ such that
$$V\cap \big(\big[2^{c}\big]^{-1}\big(V^{\sigma}+P\big)\subset V\cap Z_{1,P}\subsetneq V.$$
Moreover, the corresponding union in $V'$ consists of $(4c)^{2\dim(B)}$ irreducible components, each giving rise to a hypersurface of this kind.

By Lemma~\ref{lemma:interpolation}\ref{item:interpolation1}, for each $P\in\varphi_V^{-1}(B[2] \setminus \{0\})$,
there also exists a hypersurface $Z_{2,P}$ of degree at most $6ng \, \delta_0(V)$ such that
$$ V\cap \big(V+P\big)\subset V\cap Z_{2,P}\subsetneq V.$$ Moreover, the corresponding union in $V'$ consists of $2^{2\dim(B)}-1$ such varieties, each giving rise to a hypersurface if this kind.

By the same argument, for each $P\in\varphi_V^{-1}(B[2])$, we obtain a hypersurface $Z_{3,P}$ of degree at most $6ng \, \delta_0(V)$ such that $$ V\cap \big(V^{\rho}+P\big)\subset V\cap Z_{3,P}\subsetneq V,$$ and there are $2^{2\dim(B)}$ hypersurfaces of this kind.

So we let
\[
Z:=\bigcup_{\substack{P \in \varphi_V^{-1}(B[4c]) }}  Z_{1,P}
\quad 
\cup \bigcup_{P \in {\varphi_V^{-1}(B[2]\setminus \{0\})}} Z_{2,P}
\quad 
\cup \bigcup_{P \in {\varphi_V^{-1}(B[2])}} Z_{3,P},
\]
which satisfies $V_{\mathrm{tors}} \subset V\cap V'\subset V\cap Z\subsetneq V$ and has degree
\begin{eqnarray*}
\deg(Z) & \leq & \big((4c)^{2g}2^{2gc+2}+6(2^{2g}-1)+6 \cdot2^{2g}\big) ng \, \delta_0(V)\\
& \leq & (c^{2g}2^{2gc+4g+2}+ 2^{2g+4}) \,ng\,\delta_0(V)\\
& \leq & c^{2g} 2^{2gc+4g+3}\,ng\,\delta_0(V) \leq 2^{4cg+2g+3}\,ng\,\delta_0(V) \\
& \leq & 16^{g(c+1)}\,n\,\delta_0(V).
\end{eqnarray*}
This completes the proof of our proposition.
\end{proof}


\subsection{Geometric preparation}

We are now almost in a position to complete the proof of our Theorem~\ref{main theorem}. We first need two technical results of geometric flavour. The first one is a weighted variant of B\'ezout's theorem and is due to Philippon (\cite{Philippon95}, Corollaire~5).

\begin{lemma}\label{lemma:HA3}
Let $X$ be an irreducible variety and $Z_1, \ldots, Z_t$ hypersurfaces of~$\PP^n$ with degree $\leq \theta$.
For $1 \leq s \leq t$, if $X_s:=X\cap Z_1\cap\cdots\cap Z_s$, we have
$$
\sum_{\substack{W\subset X_s\\ \mbox{\scriptsize{irred. comp.}}}}\deg(W)\theta^{\dim(W)}
\leq
\deg(X)\theta^{\dim(X)}.
$$
\end{lemma}
\begin{proof}
We prove this by induction. For $s=1$, this is a consequence of the theorem of B\'ezout. Assume that the result holds for $s\geq 1$, and take $W$ to be an irreducible component of~$X_s$.
Since the inequality in the statement is a sum over the irreducible components, it is enough to prove
$$\sum_{\substack{W'\subset W\cap Z_{s+1}\\ \mbox{\scriptsize{ irred. comp.}}}}\deg(W')\theta^{\dim(W')}\leq \deg(W)\theta^{\dim(W)}.$$
If $W\subset Z_{s+1}$, this is trivially an equality. Otherwise, by Krull's Hauptidealsatz $$\dim(W')=\dim(W)-1,$$ and the inequality follows once again from B\'ezout.
\end{proof}

A second statement concerns the existence of an obstructing hypersurface in a relative setting. This will allow us to proceed inductively to reach the torsion subsets of small dimension.

\begin{proposition}\label{prop:Theorem}
Let $W \subset V \subset A$ be subvarieties of $A$ with $k=\codim(V) \leq k'=\codim(W) \leq g-1$.
If $W$ is irreducible and not contained in any torsion subvariety of $V$, there exists a hypersurface $Z\subset\PP^n$ such that $W_{\mathrm{tors}}\subset W \cap Z\subsetneq W$ and
\[
\deg(Z) \leq \theta:=\Big(16^{g(c+1)}n\Big)^{g-k}\,\delta_1(V)
\]

\end{proposition}

\noindent
{\it Remark.} We stress that all the codimensions in this statement and the proof below concern subvarieties of $A$.

\begin{proof}
The proof is by contradiction. We build recursively a chain of varieties
\[
X_{k} \supseteq \cdots \supseteq X_{k'+1}
\]
satisfying the following properties, for every $k \leq r \leq k'+1$:
\begin{enumerate}[label=(\roman*)]
\item \label{item:ind1-1} $W \subset X_r$;
\item \label{item:ind1-2} each irreducible component of $X_r$ containing $W$ has codimension~$\geq r$;
\item\label{item:ind1-3} $\delta_1(X_r)\leq D_r:=\big(16^{g(c+1)}n\big)^{r-k}\,\delta_1(V)$
\end{enumerate}

For $r=k$, we choose $X_{k}:=V$ and quickly check that all three properties hold. Next, we assume that the variety $X_r$ is already constructed for some $r \geq k$.
We write $$X_r=W_1\cup\cdots\cup W_t,$$ where $W_1, \ldots, W_t$ are the irreducible components of~$X_r$.
By~\ref{item:ind1-1}, there exists an $s\geq 1$ such that $W \subset W_j$ if and only if $1\leq j\leq s$ (after possibly reordering the components).
By assumption, the varieties $W_1, \ldots, W_s$ are not torsion cosets. Thus, for every $j=1,\ldots,s$, by Proposition~\ref{prop:Z}, there is a hypersurface $Z_j$ such that
\[
\deg(Z_j)\leq 16^{g(c+1)}n \, \delta_0(W_j) \leq  16^{g(c+1)}n \, \delta_1(X_r)\leq D_{r+1},
\]
and
\begin{equation}\label{eq:ind1}
W_{j,\mathrm{tors}}\subset W_j\cap Z_j\subsetneq W_j.
\end{equation}
Because $W \subset W_j$, we have $W_{\mathrm{tors}} \subset Z_j$ which has degree $\leq D_{r+1}\leq \theta$. Since we proceed by contradiction, this forces $W \subset Z_j$. We define
\[
X_{r+1}:=X_{r}\cap\bigcap_{1 \leq j \leq s}Z_j.
\]
We know that $W \subset Z_j$ for all $1 \leq j \leq s$, so we have $W \subset X_{r+1}$,
and $X_{r+1}$ satisfies~\ref{item:ind1-1}.
To show that property~\ref{item:ind1-2} holds for $X_{r+1}$,
first observe that the only irreducible components of $X_{r+1}$ that contain $W$
are also irreducible components of $W_j\cap Z_1\cap\cdots\cap Z_s$ for some $j\leq s$.
By induction hypothesis, condition \ref{item:ind1-2} is true for $X_r$, so that $$\codim(W_j)\geq r$$ for every $1 \leq j\leq s$. The second strict inclusion in~\eqref{eq:ind1} gives
\[
\codim(W_j\cap Z_j)\geq r+1,
\]
and~\ref{item:ind1-2} is satisfied by $X_{r+1}$.
The inequalities
\[
\delta_1(X_{r+1})\leq\max\lbrace\delta_1(X_{r}),\deg(Z_1),\ldots,\deg(Z_s)\rbrace\leq D_{r+1}.
\]
finally show property~\ref{item:ind1-3} for $X_{r+1}$, which ends the proof of our induction. 

Now that the existence of our chain of varieties is proved, we see by \ref{item:ind1-2} that there is an irreducible component of $X_{k'+1}$ of codimension $\geq k'+1$ which contains~$W$. This is a contradiction for $W$ has codimension~$k'$.\qedhere
\end{proof}

\subsection{Proof of the main theorem}
We can now state a precise version of Theorem~\ref{main theorem}. It is proved by a combination of our geometric preliminaries and a second induction process.
\begin{theorem}\label{THEOREM}
Let $V\subset A$ be a variety of dimension~$d>0$. For every $0 \leq j \leq d$,
\[
\deg(V^j_{\mathrm{tors}})\leq c_j\;\delta(V)^{g-j},
\]
where
\[
c_j=\Big(16^{g(c+1)}n\Big)^{(g-j)d}\;\deg(A).
\]
\end{theorem}
\begin{proof}
Write $V:=X^0\cup\cdots\cup X^d$, where $X^j$ represents the $j$-equidimensional part of~$V$
for $0 \leq j \leq d$. To simplify the notation, we fix
$$\theta:=\Big(16^{g(c+1)}n\Big)^{d}\,\delta(V).$$
The key is to prove the following inequality
\begin{eqnarray}\label{eq:ind2-1}
\sum_{j=0}^d \deg(V^j_{\mathrm{tors}})\theta^j\leq \sum_{j=0}^d\deg(X^j)\theta^j.
\end{eqnarray}
To do so, we build inductively a family of varieties $(Y^d,\ldots,Y^0)$
such that, for $0 \leq r \leq d$:
\begin{enumerate}[label=(\roman*)]
\item \label{item:ind2-1} $Y^r$ is $r$-equidimensional,
\item \label{item:ind2-2}
$V_{\mathrm{tors}}^0 \cup \cdots \cup V_{\mathrm{tors}}^r \subset  X^0 \cup \cdots \cup X^{r-1} \cup Y^r,$
\item \label{item:ind2-3}
$\sum_{j=r+1}^d\deg(V^j_{\mathrm{tors}})\theta^{j-r} + \deg (Y^r)\leq \sum_{j=r}^d\deg(X^j)\theta^{j-r}$,
\item \label{item:ind2-4}
every irreducible component of $Y^r$ has non-empty intersection with $V_{\mathrm{tors}}$ and is not contained in any $V_{\mathrm{tors}}^j$, for $j > r$.
\end{enumerate}
Compared to $(X^d, \ldots, X^0)$, this family retails more adequate information on the torsion of $V$, and the degree of each variety is controlled.
We set $Y^d$ to be the union of all irreducible components of $X^d$ having non-empty intersection with $V_{\mathrm{tors}}$, and rapidly check that $Y^d$ satisfies conditions \ref{item:ind2-1}-\ref{item:ind2-4}.
Next, we assume that the variety $Y^r$ is already built for some $0<r\leq d$, and write
$$Y^r=V^r_{\mathrm{tors}}\cup W_1\cup\cdots\cup W_s,$$
where $W_1, \ldots, W_s$ are the irreducible components of $Y^r$ that do not lie in $V^r_{\mathrm{tors}}$.
Observe that if there is no such component, we can take $Y^{r-1}$ to be the union of all the irreducible components of $X^{r-1}$ satisfying~\ref{item:ind2-4} and check that \ref{item:ind2-1}-\ref{item:ind2-3} also hold.
Hence, we assume that $s \geq 1$. Moreover, since $Y^{r}$ satisfies~\ref{item:ind2-4}, the components $W_1, \ldots, W_s$ do not lie in a torsion coset of~$V$.

For each $j=1,\ldots,s$, we apply Proposition~\ref{prop:Theorem} with $V$, and $W=W_j$ which has codimension $\leq g-1$.
This gives a hypersurface $Z_j$ such that $$W_{j, \mathrm{tors}} \subset W_j\cap Z_j\subsetneq W_j,$$ and using (ii) of Lemma \ref{delta-ineq}: $\deg(Z_j) \leq \theta$.
Krull's Hauptidealsatz implies that $W_j\cap Z_j$ is either empty or $(r-1)$-equidimensional.
We then define
$$Y^{r-1}= X^{r-1} \cup \bigcup_{j=1,\ldots,s}(W_j\cap Z_j).$$
By construction, $Y^{r-1}$ satisfies properties \ref{item:ind2-1} and~\ref{item:ind2-2} for $r-1$.
Moreover, by B\'ezout's theorem we have
\[
\deg(Y^{r-1})\leq \theta\sum_{j=1}^s\deg(W_j) + \deg(X^{r-1}) \leq \theta\big(\deg(Y^r)-\deg(V^r_{\mathrm{tors}})\big)+ \deg(X^{r-1}),
\]
where the last inequality follows from the fact that $Y^r=V^r_{\mathrm{tors}}\cup W_1\cup\cdots\cup W_s$. Adding $\sum_{j=r}^d\deg(V^j_{\mathrm{tors}}) \, \theta^{j+1-r}$ on both sides, we find
$$\sum_{j=r}^d \deg(V^j_{\mathrm{tors}}) \, \theta^{j+1-r} + \deg(Y^{r-1})  \leq  \theta\big(\sum_{j=r+1}^d \deg(V^j_{\mathrm{tors}}) \, \theta^{j-r}+\deg(Y^r)\big) + \deg(X^{r-1}).$$
By property~\ref{item:ind2-3} in the induction step for~$r$
$$\sum_{j=r}^d \deg(V^j_{\mathrm{tors}}) \, \theta^{j+1-r} + \deg(Y^{r-1}) \leq \sum_{j=r-1}^d \deg(X^j) \, \theta^{j+1-r},$$
and this shows that $Y^{r-1}$ satisfies property~\ref{item:ind2-3} for $r-1$.
After possibly discarding some irreducible components (which does not affect the properties already proved), we secure the fact that $Y^{r-1}$ also satisfies~\ref{item:ind2-4} for $r-1$.
This concludes the induction.

Now, the inequality \eqref{eq:ind2-1} follows from the inclusion $V^0_{\mathrm{tors}}\subset Y^0$ together with \ref{item:ind2-3} for $r=0$. We use this and Lemma~\ref{lemma:HA3} with $X=A$ and a family of hypersurfaces of $A$ defining $V$ of degree $\leq \delta(V) \leq \theta$ to get
\begin{eqnarray*}
\sum_{j=0}^d \deg(V^j_{\mathrm{tors}})\theta^j\leq \deg(A) \, \theta^g.
\end{eqnarray*}
So for $0 \leq j \leq d$: $$\deg(V^j_{\mathrm{tors}}) \leq \deg(A) \, \theta^{g-j},$$ which yields the bounds announced in the theorem.
\end{proof}

We immediately get a proof of our corollary, with an explicit bound for the number of torsion cosets.

\begin{proof}[Proof of Corollary \ref{main corollary}]
The number $T$ of maximal torsion cosets in~$V$ is bounded in the following way:
\begin{eqnarray*}
T & \leq & \sum_{j=0}^d \deg(V_{\mathrm{tors}}^j) \leq \deg(A)\; \big(16^{g(c+1)}n\big)^{g\;d}\; \delta(V)^{g} \\
& \leq & \deg(A)^{g+1} \; \big(16^{g(c+1)}n\big)^{g\;d}\; \deg(V)^g,
\end{eqnarray*}
where $d=\dim(V)$ and the last inequality follows from Lemma \ref{delta-ineq}.
\end{proof}

\medskip

\noindent
{\it Remark.} For a good choice of our embedding,  we can get a more explicit constant (although the geometric measure of $V$ still depends on the ample line bundle). It is in fact possible to embed $A$ into a projective space of dimension $2g+1$ (see for instance \cite{Shafarevich:BAG1}, \S5.4 Theorem~9, and notice that for general $g$, this is optimal by~\cite{Barth:1970} and~\cite{VanDeVen:1975}). Furthermore, \cite{Iyer:2001}, Theorem~1.4 gives that the degree of $A$ with respect to this embedding is  $\leq 2^{2g}$. 

Composing by a suitable Veronese morphism yields a normal embedding (that can be assumed to be symmetric after translation) in a projective space of dimension~$\leq (2g+4)^3$, and we now get $\deg(A) \leq 6^{2g}$. If $V\subset A$ is a subvariety of dimension $d>0$, and $T$ is the number of maximal torsion cosets in~$V$, Theorem~\ref{main theorem} and Corollary \ref{main corollary} yield $$T \leq  \deg(V_{\mathrm{tors}}^d) \leq 16^{(c+3)dg^2} \;\delta(V)^g.$$

\subsection{Optimality}

We finally discuss the optimality of our Theorem \ref{main theorem}. We fix $K$ a number field, as well as two integers $g \geq 1$ and $0 \leq j \leq g-1$. Let also $B$ be an abelian variety of dimension $j$ defined over $K$.

\medskip

By \cite{Howe-Zhu02}, Theorem~1, there is an absolutely simple abelian variety $B'$ of dimension $g-j$ defined over $K$. Let $Z$ be a subvariety of $B'$ with codimension $1 \leq k \leq g-j$ that goes through the neutral element $0 \in B'$. Theorem \ref{raynaud} insures that the set of torsion points in $Z$ has finite cardinality, say $T$. By construction, we have $T \geq 1$. 

\medskip

Now, for a positive integer $d$, define $\phi_d : A:= B \times B' \rightarrow B'$ by $$\phi_d(x,y)= [d]y.$$ If $V:= \phi_d^{-1}(Z)$, we see that $V_{\mathrm{tors}}^j$ is a union of $T d^{2(g-j)}$ translates of $B$, hence $$\deg(V_{\mathrm{tors}}^j)= T \deg(C)d^{2(g-j)}.$$ Furthermore, we can get a set of equations defining $V$ by pulling back a given set of equations defining $Z$, so we have: $\delta(V) \leq \delta(Z)d^2$. We derive $$\deg(V_{\mathrm{tors}}^j) \geq T \frac{ \deg(B)}{\delta(Z)^g} \delta(V)^{g-j} \gg_{A,Z} \delta(V)^{g-j},$$
and as the degree of $V_{\mathrm{tors}}^j$ goes to infinity with $d$, we see that the dependence on $V$ (which has dimension $g-k \geq j$) in Theorem \ref{main theorem} cannot be improved for~$A$.

\bibliographystyle{amsplain}
\bibliography{biblio}

\end{document}